\documentclass[a4paper,11pt]{amsart}

\usepackage{amssymb}
\usepackage{amsfonts}
\usepackage{amsmath}

\usepackage{graphicx}
\usepackage[all]{xy}
\usepackage{array}

\usepackage{amsthm}

\newtheorem{theorem}{Theorem}[section]

\newtheorem{lemma}[theorem]{Lemma}
\newtheorem{proposition}[theorem]{Proposition}

\theoremstyle{definition}
\newtheorem{definition}[theorem]{Definition}
\newtheorem{framework}[theorem]{Framework}

\newtheorem{remark}[theorem]{Remark}
\newtheorem{example}[theorem]{Example}

\theoremstyle{remark}

\newcommand{\Si}{\mathfrak{S}}
\newcommand{\Ext}{{\mathrm{Ext}}}
\newcommand{\uE}{{\underline{\mathrm{Ext}}}}
\newcommand{\Tor}{{\mathrm{Tor}}}
\newcommand{\DD}{{\mathbf{D}}}

\renewcommand{\hom}{{\mathrm{Hom}}}
\newcommand{\End}{{\mathrm{End}}}
\newcommand{\Id}{{\mathrm{Id}}}

\renewcommand{\P}{\mathcal{P}}
\newcommand{\F}{\mathcal{F}}

\newcommand{\FF}{\mathcal{F}}
\newcommand{\C}{\mathcal{C}}
\newcommand{\M}{\tau}

\newcommand{\Z}{\mathbb{Z}}
\newcommand{\Fp}{\mathbb{F}_p}
\newcommand{\A}{\mathcal{A}}
\newcommand{\B}{\mathcal{B}}
\newcommand{\Cr}{{\mathrm{cr}}}
\newcommand{\Func}{{\mathrm{Func}}}
\newcommand{\V}{\mathrm{P}}
\newcommand{\ev}{\mathrm{ev}}

\newcommand{\DIGp}{\Sigma_p}
\newcommand{\INT}{\mathcal{I}}

\newcommand{\Comp}{\mathrm{Comp}}
\newcommand{\Fu}{\mathcal{F}}
\newcommand{\simpl}{\mathrm{simpl}}

\title[A functorial control of integral torsion in homology]{A functorial control of integral torsion in homology}
\author[Antoine Touz\'e]{Antoine Touz\'e} 
\date{\today}

\begin{document}

\sloppy

\begin{abstract}

We study the integral torsion of the values of strict
polynomial functors defined over the integers. We interpret some classical
homological invariants as values of
strict polynomial functors and therefore obtain estimates of the
integral torsion of these invariants, including cancellation results. 
\end{abstract}

\maketitle

\section{Introduction}

The aim of this article is to provide an explanation for some phenomena regarding the integral torsion of some homological invariants. 

An emblematic example is the stable homology of the Eilenberg-Mac Lane spaces $K(\pi,n)$ associated to an abelian group $\pi$. Cartan computed \cite[Expos\'e 11]{Cartan} the (unstable) integral homology of these spaces, and deduced from it their stable homology (the $i$-th stable homology group $\mathrm{H}\Z_i(H\pi)$ equals the unstable homology group $\mathrm{H}_{i+n}(K(\pi,n),\Z)$ if $i<n$). He observed \emph{a posteriori} the striking fact that the torsion part of the graded abelian group $\mathrm{H}\Z_*(\mathrm{H}\pi)$ contains only elements of order $p$ for prime integers $p$, whereas the unstable homology groups contain complicated torsion. For example, the graded abelian group $\mathrm{H}_{*}(K(\Z,4),\Z)$ contains elements of order $p^r$ for all $r\ge 0$ and for all prime integers $p$. 

Dold and Puppe gave \cite[Section 10]{DP2} a conceptual proof for this phenomenon, which does not require to compute $\mathrm{H}\Z_*(\mathrm{H}\pi)$. Such a result is useful because if we know \emph{a priori} that the stable homology has only prime torsion then we can recover it (by the universal coefficient theorem) from the stable homology with prime field coefficients $\mathrm{H}{\Fp}_*(\mathrm{H}\pi)$, whose computation is significantly easier (see e.g. \cite{Betley}). The argument of Dold and Puppe has been adapted to other contexts later, for example in \cite{FP}, but its range of application is quite limited. 

This article yields a method to obtain information on integral torsion, but with a wider range of application than Dold and Puppe's argument. The method decomposes into two steps.
\begin{enumerate}
\item In theorem \ref{thm-princ}, we describe the integral torsion 
of the values of strict polynomial functors in the sense of 
Friedlander and Suslin \cite{FS,SFB}.
\item  Then 
we give a (non exhaustive) list of homological invariants which
be interpreted as a value of a strict polynomial functor. In particular we obtain
information regarding their integral torsion.
\end{enumerate}

Let us describe the first step. Given a commutative ring $R$, we denote by 
$\V_R$ the category of finitely generated projective $R$-modules.
Strict polynomial functors can be thought of as functors $F:\V_R\to
R\text{-Mod}$, equipped with an additional 'strict polynomial structure', cf. section \ref{sec-2}.
There are two nonnegative integers attached to strict polynomial functors, 
which are usually easy to determine.
\begin{itemize}
\item The weight\footnote{What we call the `weight' of a strict polynomial
functor is called `degree' in \cite{FS,SFB}. However, we shall not follow the
terminology of \cite{FS,SFB} to prevent the confusion with the notion of degree
in the sense of Eilenberg and Mac Lane \cite{EML2}.} of the functor $F$, which is determined
by its strict polynomial structure. Homogeneous strict polynomial functors
of weight $s$ form an abelian category $\P_{s,R}$.
\item The degree $\deg(F)$ in the sense of Eilenberg and Mac Lane
\cite{EML2}. This notion of degree is defined for all functors $F:\V_R\to
R\text{-Mod}$, i.e. independently of the strict polynomial structure. 
\end{itemize}
For example, the $d$-th tensor $\otimes^d:M\mapsto M^{\otimes d}$ is canonically a strict polynomial functor of weight $d$, and its degree is also $d$. 
In general, the degree of a strict polynomial functor is less or equal to its
weight, but the inequality can be strict.
The difference between the weight of $F$ and its degree gives information regarding the integral
torsion of the values of $F$. For example,
we prove in section \ref{sec-3}:

\begin{proposition}\label{prop-amusette}
Let $R$ be a commutative ring, let $s$ be a positive integer, and let
$F\in\P_{s,R}$. If $\deg(F)<s$, then the values of $F$, considered as abelian groups, are torsion
groups. 
\end{proposition}

We will therefore restrict our attention to functors with torsion values. For such functors we have a decomposition in $\P_{s,R}$:
\begin{align*} F \simeq \bigoplus_{p\text{ prime }} {_{(p)}F}\;,
\end{align*}
where $_{(p)} F$ denotes the $p$-primary part of $F$, that is for all $M\in\V_R$
$$_{(p)}F(M)=\{x\in F(M)\;|\;\exists r\in\mathbb{N}\; p^rx=0\}.$$ 
Since $\deg ({_{(p)}F})\le \deg(F)$,
we can restrict our attention further to functors with values in $p$-primary abelian groups, which we call \emph{$p$-primary functors}.
Given a positive integer $s$, we denote by $\DIGp(s)$ the sum of the digits in the $p$-adic expansion of $s$, and by $\INT(p,s)$ the `interval' of integers 
\begin{align*}\INT(p,s)
&=\{n\in \mathbb{N}\;|\; n=s\mod (p-1)\;,\text{ and } \DIGp(s)\le n\le s\}\;,
\end{align*}
which can be pictured as follows (note that $\DIGp(s)$ is the smallest element of $\INT(p,s)$):
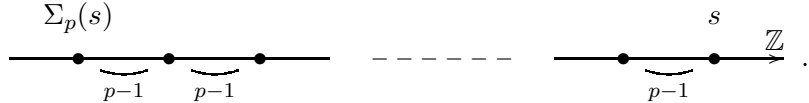
$$\xymatrix@R=1pt@W=10pt@!C=10pt{
&{\DIGp(s)} &&&&&&&s&\\
\ar@{-}[rrrr]&\bullet\ar@/_/@{-}[r]_{p-1} &\bullet\ar@/_/@{-}[r]_{p-1}&\bullet&\ar@{--}[rr]&&\ar[rrr]^>>{\text{\normalsize $\Z$}}&\bullet\ar@/_/@{-}[r]_{p-1}&\bullet&.
}$$
We say that the torsion of strict polynomial functor $F$ is bounded by an integer $n$ if for all finitely generated projective $R$-modules 
$M$ and all $x\in F(M)$ we have $nx=0$ (in particular, the torsion of $F$ is bounded by $1$ if and only if $F=0$).
In section \ref{sec-3}, we prove the following description of $p$-primary strict polynomial functors.

\begin{theorem}\label{thm-princ}
Let $R$ be a commutative ring, let $p$ be a prime integer and let $F\in\P_{s,R}$ be a \emph{nonzero} $p$-primary strict polynomial functor. Then
\begin{enumerate}
\item The degree of $F$ is an element of $\INT(p,s)$.
\item If $\deg(F)<s$, the integral torsion of $F$ is bounded by $p^r$, with $r=\big\lceil
\frac{\deg(F)+1-\DIGp(s)}{p-1}\big\rceil$, where the brackets denote the ceiling function\footnote{That is, for all $x\in\mathbb{R}$, $\lceil x\rceil=\min \{n\in\Z\;|\;x\le n\}$.}.
\end{enumerate}
\end{theorem}

In the second step, we prove that certain homological invariants are values of strict polynomial functors. 
Our method applies to the homological invariants which involve strict polynomial functors in their definition. 
This may seem very restrictive, but one should keep in mind that many classical algebraic constructions yield strict polynomial functors. 
\begin{itemize}
\item[$\bullet$] Schur and Weyl functors as defined in \cite{ABW}, are strict polynomial functors of fundamental importance in representation theory.
\item[$\bullet$] Free algebras associated to a symmetric operad in differential graded $R$-modules (free commutative algebras, free Lie algebras, etc.) yield strict polynomial functors (compare example \ref{ex-str-fct}(2) and \cite[5.1]{LodayValette}).
\item[$\bullet$] Some groups or group rings of fundamental importance are filtered, with strict polynomial functors as associated graded pieces. This is so for free groups by the Magnus-Witt isomorphism \cite{Magnus,Witt}, or for the group ring of a free abelian group (the graded pieces are symmetric powers).
\end{itemize}
The homology with coefficients $H_*(G,F(M))$ of a group $G$, where $M$ is a $RG$-module and $F\in\P_{s,R}$, provides a simple example of homological invariant depending on a strict polynomial functor. There exist less trivial examples, where the strict polynomial functors are hidden, like the stable homology of Eilenberg-Mac Lane spaces.

Given a homological invariant $\Psi(F)$ depending on a strict polynomial functor $F\in\P_{s,R}$, we introduce a parameter $M\in \V_R$ by replacing $F$ by the functor $F_M=F\circ (M\otimes_R-)$ in the construction. In this way, we obtain a `new' invariant $\Psi(F_M)$ which is a homogeneous strict polynomial functor of weight $s$ with respect to the parameter $M$. In section \ref{subsec-param}, we give conditions which allow to compute the degree of the functor $M\mapsto \Psi(F_M)$. When they apply, theorem \ref{thm-princ} gives us information on the integral torsion of $\Psi(F)=\Psi(F_R)$. We illustrate our method by the following examples.
\begin{enumerate}
\item The Taylor towers of strict polynomial functors. 
\item Derived functors of strict polynomial functors. 
\item Functor homology (computations of $\Ext$ or $\Tor$ groups in various functor categories). 
\item The cohomology of reductive algebraic groups.
\end{enumerate}

\section{Background}\label{sec-2}
This section collects some basic facts about polynomial and strict polynomial functors which will be needed in our proofs. We refer the reader e.g. to \cite{EML2,FS,Krause,PPan} for more details. 
\subsection{Polynomial functors}\label{subsec-pol}
Let $\A$ be an additive category, let $\B$ be an abelian category and let 
$F:\A\to \B$ be a functor. For all sums $X_1\oplus X_2$ in $\A$, $F(X_i)$ canonically identifies with a direct summand of $F(X_1\oplus X_2)$. We start with the definition of the cross-effects of $F$.
\begin{definition}[\cite{EML2}]
For all positive $d$, the $d$-th cross effect of $F$ is the functor $\Cr_d F:\A^{\times d}\to \B$, defined inductively as follows.
\begin{align*}
&\Cr_1 F (X)\oplus F(0)= F(X)\;,\\
&\Cr_2 F (X_1,X_2)\oplus \Cr_1 F(X_1)\oplus \Cr_2 F(X_2) = \Cr_1 F(X_1\oplus X_2)\;,
\end{align*}
And more generally, $\Cr_{d}F$ measures the non-additivity of $\Cr_{d-1}F$ with respect to its first variable:
\begin{align*}
\Cr_{d}F(X_1,\dots,X_d)\oplus &\Cr_{d-1}F(X_1,X_3,\dots,X_d)\oplus \Cr_{d-1}F(X_2,X_3,\dots,X_d)\\ 
&= \Cr_{d-1}F(X_1\oplus X_2,X_3\dots, X_d)\;.\\
\end{align*}
\end{definition}

One proves by induction the canonical decomposition:
\begin{align}F(X_1\oplus\dots\oplus X_d)= F(0)\oplus\left(\bigoplus_{k=1}^d\left(\bigoplus_{j_1<\dots<j_k} \Cr_{k}F(X_{j_1},\dots,X_{j_k})\,\right)\,\right) \;.\label{eqn-decomp}\end{align}
By letting the symmetric group $\Si_d$ act on the expression above by permuting the terms of the coproduct, we see that the $d$-th cross effect is symmetric: 
$$\Cr_d F(X_1,\dots, X_d)\simeq \Cr_d F(X_{\sigma(1)},\dots, X_{\sigma(d)})\;.$$
It also follows from the definition that $\Cr_d F$ is \emph{reduced}, i.e. it is zero as soon as one of its arguments is equal to zero. We denote by $\Func_*(\C^{\times d},\A)$ the category of reduced functors. The $d$-th cross effect yields an exact functor:
$$\Cr_d: \Func(\A,\B)\to \Func_*(\A^{\times d},\B)\;.$$
Precomposition by the diagonal functor $\Delta_d:\A\to \A^{\times d}$, $X\mapsto (X,\dots,X)$ yields an exact functor:
$$\Delta^*_d : \Func_*(\A^{\times d},\B)\to \Func(\A,\B)\;.$$
One easily proves that for all positive $d$,
$\Cr_d$ and $\Delta^*_d$ are adjoint on both sides.

\begin{definition}[\cite{EML2}]\label{def-deg}
A functor $F:\A\to \B$ is polynomial of degree less or equal to $d$ if $\Cr_{d+1}F$ is zero. It is polynomial of degree $d$ if $\Cr_{d+1}F$ equals zero and $\Cr_{d}F$ is non zero. By convention the zero functor has degree $-\infty$. We denote by $\deg(F)$ the degree of a functor $F$.
\end{definition}

In particular, a functor of degree zero is constant and a reduced functor of degree $1$ is an additive functor. Moreover, $F$ is polynomial of degree $d$ if and only if $\Cr_{d}F$ is nonzero and additive with respect to each of his arguments. Since the $d$-th cross effect functor is exact, any subfunctor or quotient of a degree $d$ functor has degree less or equal to $d$. 

\begin{definition} Let $\V_R$ be the category of finitely generated projective $R$-modules. We denote by $\FF_R$ the category of functors from $\V_R$ to all $R$-modules.  
\end{definition}
We have the following classical examples of polynomial functors in $\FF_R$.
\begin{example} Let $R$ be a commutative ring.
\begin{enumerate}
\item If $F,G\in\FF_R$ are polynomial functors of degree $d$, resp. $e$, the tensor product $F\otimes G:M\mapsto F(M)\otimes_R G(M)$ is polynomial of degree $d+e$.
\item The $d$-th tensor product functor $\otimes^d:M\mapsto M^{\otimes d}$, the $d$-th symmetric power functor $S^d:M\mapsto S^d(M)=(M^{\otimes n})_{\Si_d}$, the $d$-th exterior power functor $\Lambda^d:M\mapsto \Lambda^d(M)$ and the $d$-th divided power functor $\Gamma^d:M\mapsto \Gamma^d(M)=(M^{\otimes d})^{\Si_d}$ are examples of polynomial functors of degree $d$ of the category $\FF_R$.
\end{enumerate}
\end{example}
 
The fact that $\Delta_d^*$ is left adjoint to $\Cr_d$ implies a useful vanishing lemma.
\begin{lemma}\label{lm-vanish}
Let $G\in\FF_R$ be a polynomial functor of degree $d$. Assume that $(F_i)_{1\le i\le d+1}$ are reduced functors. Then 
$$\hom_{\FF_R}(F_1\otimes\dots\otimes F_{d+1},G)=0 \;.$$
\end{lemma}

\subsection{Strict polynomial functors}\label{subsec-str-pol}
Let $R$ be a commutative ring and let $M\in\V_R$. The symmetric group $\Si_d$ acts on $M^{\otimes d}$ by permuting the factors of the tensor product.  
The Schur category $\Gamma^s \V_R$ is the $R$-linear category defined as
follows. It has the same objects as $\V_R$, the morphisms from $M$ to $N$ the
$\Si_s$-equivariant $R$-linear maps from $M^{\otimes s}$ to $N^{\otimes s}$, and
the composition law is the composition of equivariant morphisms (if $s=0$, we
let $M^{\otimes 0}=R$, and $\Si_0=\{1\}$).
The notation $\Gamma^s \V_R$ is justified by the following identification of
homomorphisms: 
\begin{align*}\hom_{\Gamma^s \V_R}(M,N)&=\hom_{\Si_s}(M^{\otimes s}, N^{\otimes
s})=\hom_R(M^{\otimes s}, N^{\otimes s})^{\Si_s}\\
&\simeq (\hom_R(M, N)^{\otimes s})^{\Si_s}=\Gamma^s(\hom_R(M,N))
\;. \end{align*}
\begin{definition} 
The category of homogeneous strict polynomial functors of weight $s$, denoted by
$\P_{s,R}$,  is the abelian category of $R$-linear functors from $\Gamma^s\V_R$
to $R$-modules. 
\end{definition}

\begin{remark}
This presentation of $\P_{s,R}$ follows \cite{Bous, PPan}. 
The category $\P_{s,R}$ is equivalent to the category of `homogeneous strict
polynomial functors of degree $s$' introduced by Friedlander and Suslin in
\cite{FS}. We already have a notion of degree recalled in definition
\ref{def-deg}, so we rather call the integer $s$ the \emph{weight} of the strict
polynomial functors to avoid confusions.
\end{remark}

There is a functor $\gamma_s:\V_R\to \Gamma^s\V_R$ which is the identity on
objects, and whose action on morphisms is given by $\gamma_s(f)=f^{\otimes s}$.
Precomposition by $\gamma_s$ yields a forgetful functor
$$\gamma_s^*:\P_{s,R}\to \F_R\;.$$
The following properties of the forgetful functor are an easy check.
\begin{lemma}\label{prop-oubli}
The forgetful functor $\gamma_s^*$ is faithful. Let $f:F\to G$ be a morphism of
strict polynomial functors. Then $f$ is a monomorphism (resp. epimorphism, resp.
isomorphism) if and only if $\gamma_s^*(f)$ is.
\end{lemma}
Thus, the notion of strict polynomial functors can be thought of as an 
enrichment of the usual notion of functor from $\V_R$ to $R$-modules. If $F$ is
a strict polynomial functor we call $F\circ \gamma_s$ the `underlying ordinary
functor'. We often abuse notations and denote by the same letter a strict
polynomial functor and the underlying ordinary functor.

\begin{example} We have the following examples of homogeneous strict polynomial
functors of weight $s$.
\begin{enumerate}\label{ex-str-fct}
\item[(1)]
We denote by $\otimes^s$ the strict polynomial functor which sends an projective
module $M\in\V_R$ to the tensor product $M^{\otimes s}$, and whose effect on
morphisms is just given by the canonical inclusion
$$\hom_{\Si_s}(M^{\otimes s},N^{\otimes s})\hookrightarrow \hom_R(M^{\otimes
s},N^{\otimes s}) \;.$$

\item[(2)] The symmetric group $\Si_s$ acts on $\otimes^s$ by permuting the factors
of the tensor product. We denote by $\Gamma^s$ the intersection of the kernels
of the maps $\sigma-\Id:\otimes^s\to\otimes^s$, for $\sigma\in\Si_s$. 
More generally, a $\Si_s$-module $M$ yields strict polynomial functors
$(M\otimes_R \otimes^s)^{\Si_s}$ and $M\otimes_{R\Si_s}\otimes^s$.
\item[(3)] If $M$ is a finitely generated projective $R$-module, we let
$\Gamma^{s,M}$ be the strict polynomial functor defined by $\Gamma^{s,M}(N)=
\hom_{\Gamma^s\V_R}(M,N)$. In particular, $\Gamma^{s,R}$ is isomorphic to
$\Gamma^s$.
\end{enumerate}
\end{example}

The functors $\Gamma^{s,M}$ generate the category $\P_{s,R}$. To be more
specific, the Yoneda lemma yields a natural isomorphism: 
$$\hom_{\P_{s,R}}(\Gamma^{s,M},F)\simeq F(M)$$
and the canonical map 
$$\bigoplus_{n\ge 1} \Gamma^{s,R^n}\otimes F(R^n)\to F $$ 
is an epimorphism. Since the ordinary functor $N\mapsto \Gamma^s(\hom_R(R^n,N))$
is polynomial of degree $s$, and quotients of polynomial functors of degree $s$
are polynomial of degree less or equal to $s$, we have:
\begin{lemma}\label{lm-deg-poids}
If $F$ is a homogeneous strict polynomial functor of weight $s$, the underlying
ordinary functor is polynomial of degree less or equal to $s$.
\end{lemma}

\begin{remark}
Not all polynomial functors of $\F_R$ lie in the
essential image of the forgetful functor. See for example remark
\ref{rk-pol-str}. 
\end{remark}

We recall the algebra structure on divided powers, which will be crucial in the sequel of the
article. The sum of all the $(s,t)$-shuffles induces a
morphism of strict polynomial functors $\mathrm{sh}_{s,t}: (\otimes^s)\otimes
(\otimes^s)\to \otimes^{s+t}$.
These morphisms endow the tensor powers with the structure of a commutative
algebra, called the shuffle algebra \cite{EML1}. 
View $\Gamma^s\otimes\Gamma^t$ and $\Gamma^{s+t}$ as a subfunctors of
$\otimes^{s+t}$. Then $\mathrm{sh}_{s,t}$ restricts to a morphism 
$\mathrm{sh}_{s,t}:\Gamma^s\otimes\Gamma^t\to \Gamma^{s+t}$. So the divided
powers form a subalgebra of the shuffle algebra. One checks the following
properties from the definition of the divided power algebra.
\begin{lemma}\label{lm-prop-Gamma}
\begin{enumerate}
\item
Let $(s_1,\dots,s_n)$ be a $n$-tuple of positive integers of sum $s$. Then the
composite, where the first map is the canonical inclusion and the second one is
induced by the multiplication 
$$\Gamma^s\hookrightarrow\Gamma^{s_1}\otimes\dots \otimes \Gamma^{s_n}\to
\Gamma^s\;, $$
is equal to the multiplication by the multinomial
$\binom{s}{s_1,\dots,s_n}=\frac{s!}{s_1!\dots s_n!}$.
\item The multiplication of the divided power algebra induces for all
$M,N\in\V_R$ an isomorphism
$$\bigoplus_{t=0}^s\Gamma^t(M)\otimes\Gamma^{s-t}(N)\simeq\Gamma^s(M\oplus N)\;.
$$
\end{enumerate}
\end{lemma}

\section{Degrees, weights and integral torsion}\label{sec-3}

\subsection{The classification of additive strict polynomial functors}
Given a positive integer
$s$, we denote by $Q_R^s\in\P_{s,R}$ the cokernel of the morphism induced by the
multiplications (if $s=1$, the sum on the left is zero, so $Q_R^1=\Gamma^1$ is
the identity functor): 
\begin{align}\bigoplus_{0<k<s}\Gamma^{k}\otimes \Gamma^{s-k}\to
\Gamma^{s}\;.\label{eq-morph}\end{align}

\begin{lemma}
The functor $Q^s_R$ is additive.
\end{lemma}
\begin{proof}
By lemma \ref{lm-prop-Gamma} the multiplication induces an isomorphism
$\bigoplus_{t=0}^s\Gamma^t(M)\otimes\Gamma^{s-t}(N)\simeq \Gamma^{s}(M\oplus
N)$. Thus, to prove the lemma, it suffices to prove that
$\bigoplus_{t=1}^{s-1}\Gamma^t(M)\otimes\Gamma^{s-t}(N)$ is in the image of the
map \eqref{eq-morph}, which is obvious.
\end{proof}

\begin{lemma}\label{lm-univ}
Let  $F\in \P_{s,R}$ be additive. There is an isomorphism of strict polynomial
functors $F\simeq Q^s_R\otimes F(R)$.
\end{lemma}
\begin{proof}

By the Yoneda lemma there is a morphism $\phi:\Gamma^s\otimes F(R)\to F$, which
induces an isomorphism $\phi_R:\Gamma^s(R)\otimes F(R)\xrightarrow[]{\simeq}
F(R)$. We claim that the following composite is zero (where the morphism on the
left is induced by the multiplication of the divided power algebra):
$$\bigoplus_{0<k<s}\Gamma^{k}\otimes \Gamma^{s-k}\otimes F(R)\to
\Gamma^{s}\otimes F(R)\xrightarrow[]{\phi} F\;.$$
Indeed, the forgetful functor is faithful so it suffices to prove 
that the same composite, viewed in $\F_R$, is zero. Since $F$ is additive, the
composite is zero by lemma \ref{lm-vanish}.
Thus, $\phi$ induces a morphism of additive functors
$\overline{\phi}:Q_R^s\otimes F(R)\to F$, and we have a commutative diagram of
$R$-modules
$$\xymatrix{
\Gamma^s(R)\otimes F(R)\ar[rrd]^-{{\phi}_R}_-{\simeq}\ar@{->>}[d]\\
Q_R^s(R)\otimes F(R)\ar@{->>}[rr]_-{\overline{\phi}_R} &&F(R)\;.
}$$
In particular $\overline{\phi}_R$ is an isomorphism. 
By additivity, this implies that $\overline{\phi}$ is an isomorphism of strict
polynomial functors.
\end{proof}

\begin{lemma}\label{lm-val-R}
Let $s\ge 2$. If $s$ is not a power of a prime integer then $Q_R^s$ is zero. If
$s$ is a power of a prime integer $p$, the $R$-module $Q_R^s(R)$ is isomorphic
to $R/pR$.
\end{lemma}
\begin{proof} For all $i$, there is a canonical isomorphism $\Gamma^i(R)\simeq
R$, and the map
$$R\simeq \Gamma^k(R)\otimes \Gamma^{s-k}(R) \xrightarrow[]{\mathrm{mult}}
\Gamma^s(R)\simeq R$$
is equal to the multiplication by $\binom{s}{k}$. Thus $Q_R^s(R)$ is isomorphic
to $R/nR$, where $n$ is the greatest common denominator of the integers
$\binom{s}{k}$, $0< k<s$. By elementary arithmetics (see e.g. \cite[Hilfsatz 10.10]{DP2} or our lemma \ref{lm-bound}), $n$ equals $p$ if $s$ is
a power of a prime $p$, and $1$ otherwise. In the latter case, since
$Q^s_R(R)=0$ so $Q^s_R=0$ by additivity.  
\end{proof}

When $s=p^r$ is a power of a prime $p$ (with $r>0$), one would like a more
concrete description of $Q^s_R$. For this, we use the Frobenius twist functors
defined in \cite{SFB}. Given a $\Fp$-algebra $A$, the
$r$-th Frobenius twist functor $I^{(r)}_A \in\P_{p^r,A}$ is an additive
functor, whose values can be concretely described as follows. The $A$-module
$I^{(r)}_A(M)$ is the quotient of the free $A$-module on $M$, modulo the
relations (for all $m,n\in M$ and $a\in A$): 
$$(m+n)^{(r)}=m^{(r)}+n^{(r)}\;, \qquad(am)^{(r)}=a^{p^r}m^{(r)}\;,$$
where $m^{(r)}$ denotes the basis element of the free $A$-module on $M$ indexed
by $m\in M$. In particular, $I^{(r)}_A(A)\simeq A$. 

We extend the definition of Frobenius twists functors to arbitrary
commutative ground  rings $R$ in the following way. Given a prime integer $p$,
we define the $r$-th Frobenius twist functor $I^{(r)}_{p,R}\in\P_{p^r,R}$ as the
composite:
$$ \Gamma^{p^r}\V_R\xrightarrow[]{\otimes
\Fp}\Gamma^{p^r}\V_{R/pR}\xrightarrow[]{I^{(r)}_{R/pR}} R\text{-Mod}\;.$$
In particular $I^{(r)}_{p,R}(M)=I^{(r)}_{R/pR}(M/pM)$, so $I^{(r)}_{p,R}$
is additive and $I^{(r)}_{p,R}(R)$ is isomorphic to $R/pR$.

\begin{lemma}\label{lm-descr}
Let $p$ be a prime and let $r$ be a positive integer. Then $Q^{p^r}_R$ is
isomorphic to the $r$-th Frobenius twist functor $I^{(r)}_{p,R}$. 
\end{lemma}
\begin{proof}
Since $Q^{p^r}_R(R)=R/pR$ we have an isomorphism $Q^{p^r}_R\simeq
Q^{p^r}_R\otimes R/pR$. The result now follows from
lemma \ref{lm-univ}, applied to $F=I^{(r)}_{p,R}$. 
\end{proof}

Gathering the results of lemmas \ref{lm-univ}, \ref{lm-val-R} and
\ref{lm-descr}, we obtain the following classification of additive strict
polynomial functors over an arbitrary ring $R$.

\begin{proposition}\label{prop-classif-add}
Let $R$ be a commutative ring, let $s$ be a positive integer and let
$F\in\P_{s,R}$ be additive. 
\begin{enumerate}
\item If $s=1$, there is an isomorphism $F\simeq \Id\otimes F(R)$.
\item If $s\ge 2$, there exists a prime $p$ and a positive integer $r$ such that
$s=p^r$, and an isomorphism $F\simeq I^{(r)}_{p,R}\otimes F(R)$. In
particular the values of $F$ only consist of $p$-torsion elements.
\end{enumerate}
\end{proposition}

\begin{remark}\label{rk-pol-str}Many additive
functors of $\F_R$ do not come from strict polynomial functors. For example, we
can define for each ring morphism $\phi:R\to R$ an
additive functor
$I^{\phi}$, which sends a $M\in\V_R$ to the quotient of the free $R$-module
with basis $(b_m)_{m\in M}$ by the relations
$b_{m+n}=b_m+b_n$ and  $b_{rm}=\phi(r)b_m$.
The Frobenius twist functors are of this form, but there might be other ring
morphisms, which yield different additive functors.
\end{remark}

\subsection{Weight vs degree for $p$-primary functors}\label{subsec-w-vs-d}

In this section, we prove theorem 
\ref{thm-princ}(1).  The case of additive (or
degree one) functors is already handled by proposition \ref{prop-classif-add}.
To study the case of higher degree functors, we need to lift the cross effect
functors $\Cr_n$ to the level of the categories of strict polynomial (multi)functors.
This is already used in \cite{Bous,FFSS,SFB,TouzeClassical}. We briefly recall how it
works.

For $n\ge 1$, the Schur category $\Gamma^s(\V_R^{\times n})$ is the $R$-linear category defined as follows. It has the same objects as $\V_R^{\times n}$. The $R$-module of morphisms are  defined by 
$$\hom_{\Gamma^s(\V_R^{\times n})}(M,N)=\Gamma^s\left(\hom_{\V_R^{\times n}}(M,N)\right)\;.$$
The composition law $\circ$ is the unique $R$-linear morphism fitting into the following diagram
$$\xymatrix{
\Gamma^s(\hom_{\V_R^{\times n}}(N,P))\otimes \Gamma^s(\hom_{\V_R^{\times n}}(M,N))\ar[r]^-{\circ}\ar@{^{(}->}[d]& \Gamma^s(\hom_{\V_R^{\times n}}(M,P))\ar@{^{(}->}[d]\\
(\hom_{\V_R^{\times n}}(N,P)\otimes\hom_{\V_R^{\times n}}(M,N))^{\otimes s}\ar[r] &\hom_{\V_R^{\times n}}(M,P)^{\otimes s}
}$$
where the vertical maps are the canonical inclusions and the bottom horizontal map is induced by the composition in $\V_R^{\times n}$. If $n=1$, this is the same as the Schur category defined in section \ref{subsec-str-pol}. As in the case $n=1$, there is a functor $\gamma_s:\V_R^{\times n}\to \Gamma^s(\V_R^{\times n})$, which is the identity on objects, and whose effect on morphisms is given by $\gamma_s(f)=f^{\otimes s}$.
\begin{definition}
The category $\P_{s,R}(n)$ of strict polynomial $n$-functors of total weight $s$ is the category of $R$-linear functors from $\Gamma^s(\V_R^{\times n})$ to $R$-modules. The functor $\gamma_s$ induces a forgetful functor 
$$\gamma_s^*:\P_{s,R}(n)\to \Func_*(\V_R^{\times n},R\text{-Mod})\;.$$
\end{definition}

To extend cross effects to the category of strict polynomial functors, we use the following two observations. First, there is a functor $\Gamma^s(\V_R^{\times n})\to \Gamma^s(\V_R^{\times n-1})$ which sends $(M_i)_{1\le i\le n}$ to $(M_1\oplus M_2, M_3\dots M_n)$ (and which is the canonical inclusion on morphisms). Thus, if $F\in\P_{s,R}(n-1)$ then 
$$F(M_1\oplus M_2,M_3\dots,M_n)$$
can be seen as an element of $\P_{s,R}(n)$, with variables $M_1,\dots,M_n$. 
Moreover, if $m\le n$, the canonical projections $\V_R^{\times n}\to \V_R^{\times m}$ and inclusions $\V_R^{\times m}\to \V_R^{\times n}$ on a chosen set of $m$ factors of $\V_R^{\times n}$ induce functors between the corresponding Schur categories. In particular, for $i=1,2$ the morphisms induced by the canonical projection from $M_1\oplus M_2$ to $M_i\oplus 0$ induces a morphism of strict polynomial $n$-functors
$$F(M_1\oplus M_2,M_3,\dots,M_n)\to F(M_i\oplus 0,M_3,\dots,M_n) \;.$$
Thus, there is a well defined functor $\Cr_n:\P_{s,R}\to \P_{s,R}(n)$ fitting into a commutative diagram
$$\xymatrix{
\P_{s,R}\ar[r]^-{\Cr_n}\ar[d]^{\gamma_s^*} & \P_{s,R}(n)\ar[d]^{\gamma_s^*}\\
\F_{R}\ar[r]^-{\Cr_n} & \Func_*(\V_R^{\times n},R\text{-Mod})\;.
}$$

To analyze cross-effects of strict polynomial functors, we use the subcategories of homogeneous strict polynomial functor with weights the $n$-tuple $(s_1,\dots,s_n)$ of nonnegative integers with $\sum_{i=1}^n s_i=s$. To be more specific, let 
$\Gamma^{(s_1,\dots,s_n)}(\P_R^{\times n})$ denote the $R$-linear category with the same objects as $\V_R^{\times n}$, whose morphisms from $M=(M_i)_{1\le i\le n}$ and $N=(N_i)_{1\le i\le n}$ are the $\prod_{i=1}^n\Si_{s_i}$-equivariant morphisms from $\bigotimes_{i=1}^n(M_i)^{\otimes s_i}$ to $\bigotimes_{i=1}^n(N_i)^{\otimes s_i}$, and whose composition law is composition of equivariant morphisms. (By convention, $M^{\otimes 0}=R$ and $\Si_0=\{1\}$ in this definition). 

\begin{definition}
The category $\P_{s_1,\dots,s_n,R}(n)$ of strict polynomial $n$-functors of weight $(s_1,\dots,s_n)$ is the category of $R$-linear functors from $\Gamma^{(s_1,\dots,s_n)}(\V_R^{\times n})$ to $R$-modules.  
\end{definition}

There is a canonical isomorphism, compatible with composition of morphisms, where the sum is taken over all $n$-tuples $(s_1,\dots,s_n)$ of nonnegative integers such that $\sum s_i=s$:
\begin{align}\hom_{\Gamma^{s}(\V_R^{\times n})}(M,N)\simeq
\bigoplus\hom_{\Gamma^{(s_1,\dots,s_n)}(\V_R^{\times
n})}(M,N)\;.\label{eq-decomposition}\end{align}
Thus we have the following functorial analogue of the usual decomposition of modules over a ring with a finite set of central idempotents.
\begin{lemma}\label{lm-decomp}
The category $\P_{s_1,\dots,s_n,R}(n)$ identifies to an abelian subcategory of $\P_{s,R}(n)$ and there is a decomposition, where the sum is taken over all $n$-tuples of nonnegative  integers $(s_1,\dots,s_n)$ such that $\sum_{i=1}^n s_i=s$: 
$$\P_{s,R}(n)=\bigoplus \P_{s_1,\dots,s_n,R}(n)\;.$$
\end{lemma}

If we partially evaluate a strict polynomial 
$n$-functor $F\in\P_{s_1,\dots, s_n,R}(n)$ on a family
$(M_1,\dots,M_{i-1},M_{i+1},\dots,M_n)\in \V_R^{\times n-1}$, we get a
homogeneous strict polynomial functor of weight $s_i$ of the remaining variable.
In particular, we obtain the following result.

\begin{proposition}\label{prop-degree}
Let $F\in\P_{s,R}$ be a nonzero $p$-primary functor of degree $d$. There exists $d$ nonnegative integers $r_i$ such that $\sum_{i=1}^{d} p^{r_i}= s$.
\end{proposition}
\begin{proof}
Since $F$ has degree $d$, its $d$-th cross effect is a nonzero element of $\P_{s,R}(d)$. Let $G$ be a nonzero homogeneous summand of  weight $(s_1,\dots,s_d)$ of
$\Cr_{d} F$. Then $G$ is additive with respect to each of his variables (indeed $\Cr_{d} F$ is so because $\Cr_{d+1} F=0$). By the classification of additive functors in proposition \ref{prop-classif-add}, this implies that each $s_i$ has the form $p^{r_i}$ for some nonnegative $r_i$, whence the result.
\end{proof}

The following elementary lemma, together with proposition \ref{prop-degree}, provides a proof of the first statement of theorem \ref{thm-princ}. 

\begin{lemma}\label{lm-carries}
Let $d$ be a positive integer. Then $d\in\INT(p,s)$ if and only if there exists $d$ nonnegative integers $r_i$ such that $\sum_{i=1}^{d} p^{r_i}= s$.
\end{lemma}
\begin{proof}
If $x$ and $y$ are positive integers, we have $\DIGp(x+y) + c(p-1) = \DIGp(x)+\DIGp(y) $, where $c$ is the number of carries which appear when we add $x$ and $y$, using $p$-adic expansions. 
In particular, if $s=\sum_{i=1}^d p^{r_i}$, there exists a nonnegative integer $c'$ such that 
$\DIGp(s)+c'(p-1)=d$. Thus $d\in\INT(p,s)$.
Conversely, we prove by induction on $k$ that for all $\DIGp(s)+k(p-1)\in \INT(p,s)$, there exists a decomposition of $s$ as a sum of $\DIGp(s)+k(p-1)$ powers of $p$. The decomposition for $k=0$ is given by the $p$-adic decomposition of $s$. If $\DIGp(s)+k(p-1)<s$ and $s=\sum p^{r_i}$ for $\DIGp(s)+k(p-1)$ nonnegative integers $r_i$, then one of the $r_i$ is strictly positive, and by replacing $p^{r_i}$ by $p$ terms $p^{r_i-1}$, we obtain that the property is valid for $k+1$.
\end{proof}

\subsection{The integral torsion of strict polynomial functors}
In this section, we prove theorem \ref{thm-princ}(2) and proposition \ref{prop-amusette}.
A \emph{composition} of $s$ into $n$ parts is a $n$-tuple of positive integers $\lambda=(\lambda_1,\dots,\lambda_n)$ with $\sum\lambda_i=s$.
We denote by $\Comp(s,n)$ the finite set of all compositions of $s$ into $n$ parts. The set $\Comp(s,n)$ is empty if and only if $n>s$. 
If $n\le s$, we denote by $\M(s,n)$ the greatest common denominator of the set of the multinomials $\binom{s}{\lambda_1,\dots,\lambda_n}$, for all compositions $(\lambda_1,\dots,\lambda_n)$ of $s$ into $n$ parts:
$$\M(s,n)= g.c.d.\left\{ \left.\binom{s}{\lambda_1,\dots,\lambda_n} \quad\right|\quad(\lambda_1,\dots,\lambda_n)\in\Comp(s,n)  \right\}\;.$$
\begin{definition}
Let $R$ be a commutative ring. Let $d$ and $s$ be two positive integers satisfying $d<s$. 
We denote by $C_{s,d+1}\in\P_{s,R}$ the cokernel of the map induced by the multiplication of th divided poser algebra
$$\bigoplus_{\lambda\in\Comp(s,d+1)} \Gamma^\lambda\xrightarrow[]{\;\sum \mathrm{mult}}\Gamma^s\;. $$
\end{definition}

\begin{lemma}
The functor $C_{s,d+1}\in \P_{s,R}$ takes values in torsion abelian groups. Its integral torsion is bounded by $\M(s,d+1)$. 
\end{lemma}
\begin{proof}
Let $\lambda=(\lambda_1,\dots,\lambda_{d+1})$ be a composition of $s$, and let $n_\lambda:=\binom{s}{\lambda_1,\dots,\lambda_{d+1}}$.
Then the following composite is zero.
$$\Gamma^p\xrightarrow[]{\times n_\lambda} \Gamma^p\twoheadrightarrow C_{s,d+1}\;\qquad (*)$$
Indeed, by lemma \ref{lm-prop-Gamma}, multiplication by $n_\lambda$ is equal to the composite
$\Gamma^s\hookrightarrow \Gamma^\lambda\xrightarrow[]{\mathrm{mult}}\Gamma^s$, and by definition of $C_{s,d+1}$, the composite $\Gamma^\lambda\to\Gamma^s\twoheadrightarrow C_{s,d+1}$ equals zero. Since $(*)$ is zero, the torsion of $C_{s,d+1}$ is bounded by $n_\lambda$, for all composition $\lambda$. 
Hence the torsion of  $C_{s,d+1}$ is bounded by $\M(s,d+1)$.
\end{proof}

Our next task is to give an explicit computation of $\M(s,n)$.
Given a nonnegative integer $x$, we denote by $N_p(x)$ the exponent of the
prime $p$ in the prime decomposition of $x$. In particular we have 
$$N_p(\M(s,n))=\min\left\{\left. N_p\left(\binom{s}{\lambda_1,\dots,\lambda_n}\right)\;\right|\lambda\in\Comp\big((s),n\big)  \right\}\;. $$

\begin{lemma}\label{lm-bound} Let $s,n$ be positive integers, satisfying $n\le s$.
$$N_p(\M(s,n))=\left\{\begin{array}{ll}0  &\text{if $n\le \DIGp(s)$,}\\ \left\lceil\frac{n-\DIGp(s)}{p-1}\right\rceil & \text{ if $n> \DIGp(s)$,} \end{array}\right.$$ 
where the brackets denote the ceiling function.
\end{lemma}
\begin{proof}[Proof of lemma \ref{lm-bound}]
We first observe that if $n\le m$, then 
\begin{equation}N_p(\M(s,n))\le N_p(\M(s,m)) \label{eq-1}\end{equation}
Indeed, we can assume $m=n+1$. Let $\lambda=(\lambda_1,\dots,\lambda_{n+1})$ be a composition of $s$ such that $N_p(\M(s,n+1))=N_p(\binom{s}{\lambda})$ and let $\lambda'=(\lambda_1,\dots,\lambda_n+\lambda_{n+1})$. Then  $\binom{s}{\lambda}=\binom{\lambda_n+\lambda_{n+1}}{\lambda_n,\lambda_{n+1}}\binom{s}{\lambda'}$. So we have  $$N_p(\M(s,n))\le N_p(\binom{s}{\lambda'})\le N_p(\binom{s}{\lambda})=N_p(\M(s,n+1))\;.$$

Let $s=s_kp^k+\dots+s_1p+s_0$ be the $p$-adic expansion of $s$. By Kummer's theorem \cite[p. 32]{PRIME}, the multinomial associated to the composition $(p^k,\dots,p^k,\dots,p^0,\dots, p^0)$, where each $p^i$ has multiplicity $s_i$, is prime to $p$. Thus, $N_p(\tau(s,\DIGp(s)))=0$. Hence the inequation \eqref{eq-1} implies:
\begin{equation}N_p(\M(s,n))=0\quad \text{ if $1\le n\le \DIGp(s)$}\label{eq-2}\;. 
\end{equation}
There is only one composition on $s$ into $s$ parts, namely $(1,\dots,1)$. So by Legendre's theorem \cite[p. 32]{PRIME}, we have
\begin{equation}N_p(\M(s,s))=\frac{s-\DIGp(s)}{p-1}\label{eq-3}\;.\end{equation}

We have determined $N_p(\M(s,n))$ for extremal values of $n$. We are now going to determine how $N_p(\M(s,n))$ grows with $n$. Let $(\lambda_1,\dots,\lambda_n)$ be a composition of $s$ such that $N_p(\M(s,n))=N_p(\binom{s}{\lambda})$. We distinguish two cases.
{\bf First case:} we assume that $n\ne \DIGp(s)\mod (p-1)$. Then by lemma \ref{lm-carries}, one of the $\lambda_i$ is not a power of $p$. For such a $\lambda_i$, one can find positive integers $k,\ell$ such that $k+\ell=\lambda_i$ and $N_p(\binom{\lambda_i}{k,\ell})=0$ (use the $p$-adic expansion of $\lambda_i$ in a fashion similar to the proof of equation \eqref{eq-2}). Thus if $\lambda'=(\lambda_1,\dots,\lambda_{i-1},k,\ell,\lambda_{i+1},\dots,\lambda_n)$, we have 
$$N_p(\M(s,n))=N_p(\binom{s}{\lambda})=N_p(\binom{s}{\lambda'})\ge N_p(\M(s,n+1))\;.$$
Hence the inequation \eqref{eq-1} implies that
\begin{equation}
N_p(\M(s,n))=N_p(\M(s,n+1))\quad \text{ if $n- \DIGp(s)\ne 0\mod (p-1)$}\label{eq-4}\;. 
\end{equation}
{\bf Second case:} we assume that $n= \DIGp(s) \mod (p-1)$ with $n<s$. If one of the $\lambda_i$ is not a power of $p$, then we may argue as in the first case to obtain that 
$$N_p(\M(s,n))= N_p(\M(s,n+1))\;.$$
Assume now that all the $\lambda_i$ are powers of $p$. Since $n<s$, one of the $\lambda_i$ is not equal to $1$, i.e. is of the form $\lambda_i=p^r$ with $r>0$. By Kummer's theorem, $N_p(\binom{p^r}{p^{r-1},(p-1)p^{r-1}})$ equals $1$, so if we denote by $\lambda'$ the composition $(\lambda_1,\dots,\lambda_{i-1},p^{r-1},(p-1)p^{r-1},\lambda_{i+1},\dots,\lambda_n)$, we have
$$
N_p(\M(s,n))+1=N_p(\binom{s}{\lambda})+1=N_p(\binom{s}{\lambda'})\ge N_p(\M(s,n+1))
$$
So we have
\begin{equation}
N_p(\M(s,n+1))\le N_p(\M(s,n))+1 \quad \text{ if $n- \DIGp(s)= 0\mod (p-1)$}\label{eq-5}\;.
\end{equation}

By equations \eqref{eq-2} and \eqref{eq-3}, the difference $N_p(\M(s,s))-N_p(\DIGp(s))$ is equal to $(s-\DIGp(s))/(p-1)$. In view of \eqref{eq-4}, this implies that the inequation \eqref{eq-5} must be an equality. Lemma \ref{lm-bound} follows by induction on $n$. 
\end{proof}

The following proposition provides a proof of proposition \ref{prop-amusette} and of theorem \ref{thm-princ}(2). 

\begin{proposition}
Let $F\in\P_{s,R}$ be a strict polynomial functor of degree $d<s$. 
Then the values of $F$, considered as abelian groups, are torsion groups.
Moreover the torsion of $F$ is bounded by $\M(s,d+1)$. In particular, if $F$
takes values in $p$-primary abelian groups, the torsion of $F$ is bounded by
$p^r$, with $r=\lceil\frac{d+1-\DIGp(s)}{p-1}\rceil$.
\end{proposition}
\begin{proof}
Given functor $F\in\P_{s,R}$, the functor $F_M$ defined by $F_M(N)= F(M\otimes N)$ is also an object of $\P_{s,R}$, with the same degree as $F$. And moreover $F_M(R)=F(M)$. So it actually suffices to prove that for all functors $F\in\P_{s,R}$, $F(R)$ is a torsion abelian group, with torsion bounded by $\M(s,d+1)$.
By the Yoneda lemma, for all $x\in F(R)$ we can find a morphism $f_x:\Gamma^s\to F$ such that $x$ lies in the image of $f_x$. Since $F$ has degree $d$, the composite
$$\bigoplus_{\lambda\in\Comp(s,d+1)} \Gamma^\lambda\xrightarrow[]{\;\sum \mathrm{mult}}\Gamma^s\xrightarrow[]{f_x} F $$
is zero, hence $f_x$ factors into a morphism $\overline{f}_x:C_{s,d+1}\to F$ which contains $x$ in its image. Since the torsion of $C_{s,d+1}$ is bounded by $\M(s,d+1)$, we obtain that $\M(s,d+1)x=0$. This holds for all $x\in F(R)$ so the proposition follows.
\end{proof}

\section{Applications to the integral torsion in homology}\label{sec-4}

\subsection{Functor categories and parameterization}\label{subsec-param}
We fix a commutative ring $R$ and we denote by $\Fu$ one of functor categories $\FF_R$ or $\P_{s,R}$.
Given $F\in\Fu$ and $M\in \V_R$, we define the parameterized functors $F_M$ and $F^M$ by: 
$$F_M:N\mapsto F(M\otimes_R N)\;,\qquad F^M:N\mapsto F(\hom_R(M, N))\;.$$
Parameterization yields exact functors $-^M,-_M: \Fu\to \Fu$, which are adjoint on both sides. 
By taking degreewise parameterization of chain complexes in $\Fu$, we get endofunctors of the (unbounded) derived category, which are adjoint on both sides:
$$-^M,-_M:\DD(\Fu)\to \DD(\Fu) \;.$$
We will work in the following framework.

\begin{framework}\label{frame}
We consider an additive functor
$$\Psi:\C\to R\text{-Mod}\;.$$
where $\C$ is a full additive subcategory of $\DD(\Fu)$, which is idempotent complete (i.e. stable by direct summand) and stable by parameterization. That is, for all $M\in\V_R$, the parameterization functor $-_M$ restricts to $-_M:\C\to \C$. 
\end{framework}

Our purpose is to obtain some information on the integral torsion of the values of $\Psi$.
To this purpose, we study the functors $\Psi_C\in \Fu$, defined for all $C\in \C$ by 
$$\Psi_C(M):=\Psi(C_M)\;.$$
In the remainder of section \ref{subsec-param}, we study of the degree of the functors $\Psi_C$.

The definition of degree extends to complexes of functors. To be more specific, the $n$-th cross effect of a complex of functors is defined by taking $n$-th cross effects degreewise. This yields an exact functor (where $\Fu(n)$ is the category of functors with $n$ variables)
$$\Cr_n:\mathbf{Ch}(\Fu)\to \mathbf{Ch}(\Fu(n))$$
and also a functor between the corresponding derived categories. A complex $C$ has degree less than $d$ if $\Cr_d C\simeq 0$ in $\DD(\Fu(d))$ or equivalently if all its homology groups $H_i(C)$ are functors of degree less than $d$ in the sense of definition \ref{def-deg}.

Given a complex of functors $C$ and $M_1,\dots,M_n\in\V_R$, we introduce a complex of functors $C[M_1,\dots,M_n]$:
$$C[M_1,\dots,M_n]:N\mapsto\Cr_nC(M_1\otimes N,\dots,M_n\otimes N)\;.$$
By definition, $C[M_1,\dots,M_n]$ is a direct summand of the parameterized complex $C_{M_1\oplus\dots\oplus M_n}$. In particular, if $C$ is an object of the category $\C$ of framework \ref{frame}, then $C[M_1,\dots,M_n]$ also is. The following result is a straightforward consequence of the additivity of $\Psi$.
\begin{lemma}\label{lm-cross-effects}
Let $C\in\C$. Then for all $M_1,\dots,M_n\in\V_R$, we have:
$$\Cr_n(\Psi_C)(M_1,\dots,M_n)=\Psi(C[M_1,\dots,M_n])\;.$$ 
In particular, if $C$ is polynomial of degree less than $d$, then $\Psi_C$ is polynomial of degree less than $d$.
\end{lemma}

Lemma \ref{lm-cross-effects} gives a condition which ensures that $\Psi_C$, for a given $C\in\C$, is polynomial. 
In concrete cases, we often encounter functors $\Psi$ such that $\Psi_C$ is polynomial of degree less than $d$, for all $C\in\C$. 
Such a global polynomiality condition is linked with the vanishing of $\Psi$ on tensor products.

\begin{proposition}[Global polynomiality, necessary condition]\label{prop-cross-implies-vanish}
Assume that for all $C\in \C$, $\Psi_C$ has degree less than $d$. Then for all family of degreewise reduced complexes $C_k$, ${1\le k\le d}$, such that
$C_1\otimes\dots\otimes C_{d}\in\C$, we have
$$\Psi(C_1\otimes\dots\otimes C_{d})=0\;.$$ 
\end{proposition}
\begin{proof}
Since the complexes $C_i$ are degreewise reduced, i.e. the complexes of $R$-modules $C_i(0)$ are \emph{equal} to zero, the tensor product $C_1(M_1)\otimes \dots \otimes C_d(M_d)$ is a direct summand of $\Cr_n (C_1\otimes\dots \otimes C_n)(M_1,\dots,M_d)$. So $C_1\otimes\dots \otimes C_d$ is a direct summand of $(C_1\otimes\dots\otimes C_d)[R,\dots,R]$. By lemma \ref{lm-cross-effects}, $\Psi$ vanishes on $(C_1\otimes\dots\otimes C_d)[R,\dots,R]$, hence on $C_1\otimes\dots \otimes C_d$. 
\end{proof}

A converse of proposition \ref{prop-cross-implies-vanish} holds when $\C$ equals $\DD_{\ge 0}(\Fu)$, the full subcategory of $\DD(\Fu)$ of nonnegatively graded chain complexes, and when $\Psi$ is some homology group of a homological functor  $\DD_{\ge 0}(\Fu)\to \DD_{\ge 0}(R\text{-Mod})$ (i.e. an additive functor commuting with suspension and preserving exact triangles).

\begin{proposition}[Global polynomiality, sufficient condition]\label{prop-vanish-implies-cross}
Let $\Phi:\DD_{\ge 0}(\Fu)\to \DD_{\ge 0}(R\text{-}\mathrm{Mod})$ be a homological functor, and let $\Psi_i=H_i\Phi$. Assume that for all $i\le k$,  $\Psi_i$ vanishes on direct sums of $d$-fold tensor products of reduced functors.
Then for all $i\le k$, and all $C\in \DD_{\ge 0}(\Fu)$, $(\Psi_i)_C$ is polynomial of degree less than $d$.
\end{proposition}
\begin{proof}
Let us call a \emph{special functor} a tensor product $F_1\otimes\dots\otimes F_d\in\Fu$ where all the $F_i$ are reduced with $R$-projective values, and a \emph{special complex} a complex of direct sums of special functors.

First, for $i\le k$, $\Psi_i$ vanishes on special complexes. This follows from a dimension shifting argument. If $C$ is a special chain complex, the exact sequence of complexes $C_0\to C\to C'=C/C_0$ yields an exact triangle in $\DD_{\ge 0}(\Fu)$ and since $\Psi_i(C_0)=0$ for $0\le i\le k$, we get $\Psi_0(C)=0$, and $\Psi_{i}(C)=\Psi_i(C')=\Psi_{i-1}(C'[-1])$ for $1\le i\le k$. So the result is obtained by induction on $i$. 
Thus, by lemma \ref{lm-cross-effects}, it suffices to prove that for all complex $C$ and all $R$-modules $M_i$,  $C[M_1,\dots,M_d]$ is isomorphic in $\DD_{\ge 0}(\Fu)$ to a special complex. 

To prove this, we use the following key property. Any reduced (i.e. equal to zero as soon as one variable is equal to zero) functor in $\Fu(d)$ admits a projective resolution by projective functors of the form 
$$(N_1,\dots,N_d)\to P_1(N_1)\otimes\dots\otimes P_d(N_d)\;, \qquad(*)$$
where the one variable functors $P_i$ are reduced. Indeed, in the case $\Fu=\P_{d,R}$, the standard projectives are functors of the form $\Gamma^{s_1,X_1}\otimes\dots\otimes\Gamma^{s_d,X_d}$, and if $\Fu=\FF_R$, the standard projectives are of the form $R \hom_R(X_1,-)\otimes\dots \otimes R \hom_R(X_d,-)$ (here $R\hom_R(X,N)$ denotes the free $R$-module with basis $\hom_R(X,N)$). Moreover, if $F$ is reduced, one can use the canonical maps of $R$-modules $0\to N_i$ and $N_i\to 0$ to split off the nonreduced part in a projective resolution.  

The key property implies that the Cartan-Eilenberg projective resolution of the complex $\Cr_d C\in \mathbf{Ch}_{\ge 0}(\Fu(d))$ is a complex $D\in\mathbf{Ch}_{\ge 0}(\Fu(d))$ whose objects are of the form $(*)$. By evaluating the quasi-isomorphism $D\to C$ on the $R$-modules $M_i\otimes N$ we get a quasi-isomorphism of complexes of functors with a single variable $N$:
$$D(M_1\otimes N,\dots, M_d\otimes N)\to C[M_1,\dots,M_d](N)\;.\qquad (**)$$
The complex on the left-hand side is a special complex. This finishes the proof.
\end{proof}

\begin{remark}
The results of this section have obvious variants. For example, we can consider a \emph{contravariant} functor $\Psi$. In this case, we use upper parameterization: $\Psi_C(M):=\Psi(C^M)$ to get a functor $\Psi_C\in\Fu$.
We can also consider a functor $\Psi:\C\to \mathrm{Ab}$ where $\C$ is a subcategory of $\DD(\FF_R)$ or $\DD(\P_{s,R})$. In this case, we use free abelian groups $A$ as parameters: for all $M\in\V_R$, $F_A(M):=F(A\otimes_\Z M)$ and $F^A(M)=F(\hom_\Z(A,M))$. Therefore the functors $\Psi_C$ are functors from free abelian groups to abelian groups. In both cases, the results of the section remain \emph{mutatis mutandis} valid.
\end{remark}

\subsection{The Taylor towers of strict polynomial functors}\label{sec-4.0}
Given a functor $F\in \FF_R$ we still denote by $F$ its simplicial extension  
$$F:\simpl(\V_R)\to \simpl({R\text{-Mod}})\;.$$
The Taylor tower of this simplicial extension, as defined by Johnson and McCarthy \cite{JM}, is  a diagram of functors with source 
$\simpl(\V_R)$
and target $\simpl({R\text{-Mod}})$, of the following form:
$$\xymatrix{
\dots\ar[r]^-{q_{n+1}}&P_nF\ar[r]^-{q_{n}}&\dots\ar[r]^-{q_{2}}&P_1F\ar[r]^-{q_{1}}&P_0F\\
&&&&F\ar[u]^-{p_0}\ar[ul]^{p_1}\ar[ulll]^-{p_n}
}\;.$$
The term $P_nF$ is the $n$-th polynomial approximation of $F$.
Let us recall its definition when $F$ is a reduced functor (i.e. $F(0)=0$).
As recalled in section \ref{sec-2}, the $(n+1)$-th diagonal and the $(n+1)$-th cross effect yield a pair of adjoint functors between the categories of reduced functors:
\begin{align} \Cr_{n+1}: \Func_*(\V_R,R\text{-Mod})\rightleftarrows \Func_*(\V_R^{\times n+1},R\text{-Mod}):\Delta^*_{n+1}\;.\label{eq-adj}\end{align}
This pair of adjoints yields \cite[Chap. 8]{Weibel} a cotriple  $\perp_{n+1}=\Cr_{n+1}\circ \Delta_{n+1}^*$. Hence for all reduced functor $F\in\FF_R$, we get an augmented simplicial object $\perp_{n+1}^{*+1}F\xrightarrow[]{\epsilon} F$ in the category of reduced functors. Hence, given a simplicial object $X$, we get a bisimplicial $R$-module $\perp_{n+1}^{*+1}F(X)\xrightarrow[]{\epsilon} F(X)$. The simplicial $R$-module $P_nF(X)$ by taking first the diagonal of this bisimplicial augmented $R$-module and then the homotopy cofiber of the resulting simplicial $R$-module.

\begin{remark}
In \cite{JM}, the Taylor towers are actually defined for functors with values in chain complexes. The Dold-Kan correspondence asserts that $\simpl({R\text{-Mod}})$ is equivalent to the category of nonnegative chain complexes of $R$-modules, so the theory of \cite{JM} applies to our case. The definition given above is seen to be equivalent to the definition of \cite[p. 768]{JM} by using the Eilenberg-Zilber theorem and the fact that the homotopy cofiber of simplicial $R$-modules corresponds to the cone in the category of chain complexes.
\end{remark}

Let us fix $X\in \simpl(\V_R)$ and integers $i,n$, and let us denote by $\Psi(F)$ the $i$-th homotopy group of  $P_nF(X)$ (i.e. the $i$-th homology group of the normalized chains of $P_nF$ \cite[Chap. 8]{Weibel}).
Then $\Psi$ yields an additive functor
$$\Psi: \FF_R\to R\text{-Mod}\;. $$
Following the framework of section \ref{subsec-param}, we introduce the parameterized functor with 
parameter $M\in \V_R$:
$$\Psi_F:M\mapsto \Psi(F_M)= \pi_i\left(P_n F(M\otimes X)\right)\;.$$

\begin{lemma}\label{lm-estpol}
Let $F\in\FF_R$. The parameterized functor $\Psi_F$ is polynomial of degree less or equal to $n$. If $F$ is the underlying ordinary functor of a homogeneous strict polynomial functor of weight $s>0$, then so is $\Psi_F$.
\end{lemma}
\begin{proof}
The fact that $\Psi_F$ is polynomial of degree less or equal to $n$ follows from the fact that $P_nF$ is a degree $n$ approximation of $F$ \cite[Lm 2.11 1)]{JM}. Now as explained in section \ref{subsec-w-vs-d}, the adjunction \eqref{eq-adj} lifts on the level of strict polynomial functors. Hence the definition of the Taylor tower lifts to the level of strict polynomial functors. In particular, if $F$ is the underlying ordinary functor of a strict polynomial functor $\widetilde{F}\in\P_{s,R}$, then $\Psi_F$ is the underlying ordinary functor of the strict polynomial functor $\Psi_{\widetilde{F}}\in \P_{s,R}$ defined similarly.
\end{proof}

As a consequence of lemma \ref{lm-estpol}, proposition \ref{prop-amusette} and theorem \ref{thm-princ}, we obtain the following estimation of the integral torsion of $\pi_i\left(P_nF(X)\right)$.
\begin{theorem}\label{thm-res1}
Let $R$ be a commutative ring, let $s$ be a positive integer, let $F\in\P_{s,R}$, and let $n<s$. Then for all integer $i$ and all $X\in \simpl(\V_R)$, $\pi_i\left(P_nF(X)\right)$ is a torsion abelian group. Its $p$-primary part is zero if $n<\DIGp(s)$, and if $n\ge \DIGp(s)$, it is bounded by $p^r$, with $r=\big\lceil
\frac{n+1-\DIGp(s)}{p-1}\big\rceil$.
\end{theorem}

The homotopy fiber of the morphism $p_n:P_nF\to P_{n-1}F$ is called the $n$-th layer of the Taylor tower of $F$, and is denoted by $D_nF$ \cite[p. 772]{JM}. As another consequence of theorem \ref{thm-princ}, we prove the following lacunarity phenomenon for the homotopy groups of $D_nF$.

\begin{theorem}\label{thm-res2}
Let $R$ be a commutative ring, let $s$ be a positive integer, let $F\in\P_{s,R}$, and let $n<s$. Then for all integer $i$ and all $X\in \simpl(\V_R)$, $\pi_i\left(D_nF(X)\right)$ is a torsion abelian group. If $n\ne s\mod (p-1)$, then its $p$-primary part is zero.
\end{theorem}
\begin{proof}
By theorem \ref{thm-res1}, $ \pi_i(P_nF(X))$ and $\pi_{i+1}(P_{n-1}F(X))$ are torsion abelian groups. Hence the $\pi_i\left(D_nF(X)\right)$ is also a torsion abelian group by the long exact sequence of homotopy groups:
$$\dots\to \pi_{i+1}(P_{n-1}F(X))\xrightarrow[]{\partial} \pi_i(D_nF(X))\to \pi_i(P_nF(X))\to \dots $$

To prove that the $p$-primary part of $\pi_i\left(D_nF(X)\right)$ is zero, it suffices to prove that the localization $\pi_i\left(D_nF(X)\right)\otimes_\Z\Z_{(p)}$ is zero. Since localization is exact, this is equivalent to prove that $\pi_i(D_nG(X))$ is zero, where $G$ denotes the functor which sends $X$ to $F(X)\otimes_\Z\Z_{(p)}$.
Let us define 
$\overline{G}:=G\circ\iota_X$,
where $\iota_X:\V_R\to \simpl(\V_R)$ is the functor defined by $\iota_X(M)=M\otimes X$.
Since $\iota_X$ is additive, we have:
\begin{align*}
P_n\overline{G}=(P_nG)\circ \iota_X\;,\quad\text{ and }D_n\overline{G}=(D_nG)\circ \iota_X\;.
\end{align*}
Hence, to prove theorem \ref{thm-res2}, it suffices to prove that for all $M\in\V_R$, $D_n\overline{G}(M)$ has trivial homotopy groups if $n\ne s\mod (p-1)$.

To prove this, we use the universal property of $P_n\overline{G}$. To be more specific, let $K:\V_R\to \simpl(R\text{-Mod})$ be a functor which is polynomial of degree less or equal to $n$ (i.e. all its homotopy groups are polynomial functors of degree less or equal to $n$ in the sense of definition \ref{def-deg}), and let $\tau:\overline{G}\to K$ be a natural transformation. Then for all $M\in\V_R$, we can consider the morphisms $\tau:\overline{G}(M)\to K(M)$ and $p_n:\overline{G}(M)\to P_n\overline{G}(M)$ in the homotopy category of simplicial $R$-modules, and it it shown in \cite[Lm 2.11 3)]{BM} that there exists a unique morphism $\tau'$ in this homotopy category which fits into a commutative diagram
$$\xymatrix{
\overline{G}(M)\ar[r]^-{\tau}\ar[d]_-{p_n}& K(M)\\
P_n\overline{G}(M)\ar@{-->}[ru]_-{\exists!\,\tau'}
}.$$
By its definition \cite[p. 770]{JM}, the map $q_n:P_n\overline{G}\to P_{n-1}\overline{G}$ in the Taylor tower of $\overline{G}$ is a representative of the morphism obtained by factorization of $p_{n-1}:\overline{G}\to P_{n-1}\overline{G}$ through $p_n$. 
Given a functor $K:\V_R\to \simpl(R\text{-Mod})$, we let $\deg K$ be the supremum of the degrees (in the sense of definition \ref{def-deg}) of the functors $M\mapsto \pi_i(K(M))$. The universal property of $P_n\overline{G}$ and the construction of $q_n$ have the following two consequences:
\begin{itemize}
\item[(i)] for all integers $n$, $\deg(P_{n-1}\overline{G})\le \deg(P_{n}\overline{G})$,
\item[(ii)] $D_n\overline{G}(M)$ has trivial homotopy groups for all $M\in\V_R$, if and only if  $\deg(P_{n-1}\overline{G})= \deg(P_{n}\overline{G})$.
\end{itemize}
The statements (i) and (ii) in turn imply:
\begin{itemize}
\item[(iii)]
 if $\deg(P_n\overline{G})=d$, then the morphism $q_k:P_k\overline{G}\to P_{k-1}\overline{G}$ induces an isomorphism on the level of homotopy groups for $d<k\le n$.
\end{itemize}
Finally, the functors $M\mapsto \pi_i(P_n\overline{G}(M))$ are strict polynomial of weight $s$, of degree less or equal to $n$. Moreover, by the construction of $\overline{G}$, $\pi_i(P_n\overline{G}(M))$ is a $p$-primary functor. Hence, theorem \ref{thm-princ} implies that:
\begin{itemize}
\item[(iv)]
$\deg(P_n\overline{G})\in\;(\INT(p,s)\cap [1,n])\cup \{-\infty\}$.
\end{itemize}
To sum up, we have proved that the function
$$\begin{array}{ccc}
\mathbb{N}&\to & \{-\infty\}\cup \mathbb{N}\\
n & \mapsto & \deg(P_n\overline{G})
\end{array}
$$
is nondecreasing by (i), with values in $\INT(p,s)\cup \{-\infty\}$ by (iv). Moreover the jumps, i.e. the values for which $\deg P_n\overline{G}>\deg P_{n-1}\overline{G}$, must occur for $n\in \INT(p,s)$ by (iii). In particular, the jumps must occur for $n=s\mod (p-1)$, and by (ii), the homotopy groups of $D_n\overline{G}$ are zero for $n\ne s\mod (p-1)$. This finishes the proof.  
\end{proof}

\subsection{Derived functors} 
Given a complex $C\in \mathbf{Ch}_{\ge 0}(\FF_R)$ and a simplicial object $X\in \simpl(\V_R)$ we can form a bigraded object $C_i(X_j)$ equipped with a differential relative to the index $i$ and a simplicial structure relative to the index $j$. By taking normalized chains and the total complex, we obtain a chain complex, which we denote by $c_X(C)$. If $X$ is fixed, this defines a homological functor
$$\Phi:=c_X:\DD_{\ge 0}(\FF_R)\to \DD_{\ge 0}(R\text{-Mod})\;. $$
Following the framework of section \ref{subsec-param}, we introduce the parameterized homology groups with parameter $M\in \V_R$, as an object of $\FF_R$:
$$(\Psi_i)_C:M\mapsto H_i(\Phi_C(M))= H_i(c_{X\otimes M} C)\;.$$
\begin{lemma}\label{lm-applic-derived}
(1) If $C$ is the underlying complex of a complex in $\P_{s,R}$, then the functor
$(\Psi_i)_C$ is the underlying ordinary functor of a homogeneous strict polynomial functor of weight $s$.

(2) If $X$ is $(n-1)$-connected, then for all $i<nd$ and all $C$, the functor $(\Psi_i)_C$ is polynomial of degree less than $d$.
\end{lemma}
\begin{proof}
(1) The definition of $c_X\widetilde{C}$ above makes sense when $\widetilde{C}$ is a complex of strict polynomial functors. Hence if $C$ is the underlying complex of ordinary functors a complex $\widetilde{C}\in\mathbf{Ch}_{\ge 0}(\P_{s,R})$,  then $(\Psi_i)_C$ is the underlying functor of the strict polynomial functor $(\Psi_i)_{\widetilde{C}}\in\P_{s,R}$.

(2) The simplicial $R$-module $X$ is homotopy equivalent to a simplicial $R$-module $Y$ with $Y_i=0$ for $i< n$. So for $F\in\FF_R$, $c_X F$ is homotopy equivalent to $c_Y F$. Assume that $F$ is of the form $F_1\otimes \dots\otimes F_d$. Then the Eilenberg-Zilber theorem yields an homotopy equivalence
$$ c_Y(F_1\otimes\dots\otimes F_d)\simeq c_Y(F_1)\otimes\dots\otimes c_Y(F_d)\;. \qquad (*)$$
If the functors $F_k$ are reduced, then $(c_Y F_k)_i=0$ for $i<n$. Thus, the right-hand side of $(*)$ is a complex which is zero in degrees $i<nd$. Hence 
$$\Psi_i(F_1\otimes\dots\otimes F_d)=H_i(c_X(F_1\otimes\dots\otimes F_d))=0$$
for $i<nd$ and the result follows from proposition \ref{prop-vanish-implies-cross}.
\end{proof}

By lemma \ref{lm-applic-derived}, if $C$ is a chain complex in $\P_{s,R}$, we can apply theorem \ref{thm-princ} to the functor $M\mapsto H_i(c_{X\otimes M}C)$ to obtain the following properties of the torsion in the homology of the complex of abelian groups $c_X C$.

\begin{theorem}\label{thm-applic-derive}
Let $R$ be a commutative ring, let $C\in\mathrm{Ch}_{\ge 0}(\P_{s,R})$, and let $X$ be an $(n-1)$-connected simplicial $R$-module, which is degreewise finitely generated and projective over $R$. 
\begin{enumerate}
\item For $i<ns$, $H_i(c_X C)$ is a torsion abelian group.
\item For $i < nd$, with $\DIGp(s)\le d<s$, the $p$-primary torsion of $H_i(c_X C)$ is bounded by $p^r$ with
$r=\big\lceil
\frac{d-\DIGp(s)}{p-1}\big\rceil$.
\item In particular, for $i<n\DIGp(s)$ the $p$-primary part of  $H_i(c_X C)$ is zero and for $i<n$, $H_i(c_X C)$ is equal to zero.
\end{enumerate} 
\end{theorem}

The homology groups $H_i(c_X C)$ studied in theorem \ref{thm-applic-derive} are known under various names, depending on the complex $C$ and the simplicial object $X$ considered. To make theorem \ref{thm-applic-derive} more concrete, we give some of them below. We will denote by $K(M,n)$ a free simplicial $R$-module, whose homotopy groups are all zero but $\pi_n(K(M,n))\simeq M$.

\begin{example}[Derived functors]
For a finitely generated abelian group $A$ and a functor $F\in\FF_\Z$ (considered as a complex concentrated in degree zero), the homology groups $H_i(c_{K(A,n)} F)$ are denoted by $L_iF(A;n)$ in \cite{DP1,DP2} and called the derived functors of $F$. 

Theorem \ref{thm-applic-derive} give bounds for the integral torsion of these derived functors when $F$ is actually strict polynomial. Our result does not say anything about $L_iF(A;n)$ for $i\ge ns$, but it is not hard to show that $L_iF(A;n)=0$ for $i> ns$ and $A$ free. This can be proved directly by using an explicit model of $K(A,n)$ and cross effects of $F$. Similarly, if $A$ is not free, one can prove (using cross effects and the fact that abelian groups have two steps free resolutions) that $L_iF(A;n)=0$ for $i> (n+1)s+1$. So, almost all degrees are covered by the results above.
\end{example}

\begin{example}[Stable derived functors]
The low dimension groups $L_{i+n}F(A;n)$ for $0\le i<n$ are called the stable derived functors of $F$ (they do not depend on $n$), and denoted by $L^{\mathrm{st}}_iF(A)$. Assume that $F\in\P_{s,\Z}$ with $s>1$. As a particular case of theorem \ref{thm-applic-derive}, we obtain that $L^{\mathrm{st}}_iF(A)$ is zero if $s\ne p^r$, and is a $\mathbb{F}_p$-vector space if $s=p^r$. 

This result was proved by Dold and Puppe in the special case $F=S^s$ in \cite{DP2} (their proof does not work for an arbitrary strict polynomial functor). Since there is an isomorphism 
$\bigoplus_{s\ge 0} L_i^{\mathrm{st}}S^s(A) = H\Z_*(HA)$ \cite{DP2}, this results explains why only prime torsion occurs in the stable homology of Eilenberg-Mac Lane spaces, as alluded to in the introduction.
\end{example}

\begin{example}[Singular homology of symmetric spaces]
Let $G$ be a subgroup of the symmetric group $\Si_s$. Given a CW-complex $Y$ we can consider its $G$-symmetric product $SP^G(Y):=Y^{\times s}/G$ where 
$G$ acts by permuting the factors of the product. We denote by $S^G\in \P_{s,\Z}$ the algebraic analogue of $G$-symmetric products, namely $S^G(A)=(A^{\otimes s})_G$.
Assume that $Y$ is $(n-1)$-connected with finitely generated singular homology groups $H_i(Y)$ for all $i$ ($n\ge 2$). Then the simplicial abelian group of singular chains of $Y$ is homotopy equivalent to a $(n-1)$-connected simplicial abelian group $X\in \simpl(\P_\Z)$. Dold proved \cite[Proof of thm (7.2)]{Do} that 
$H_i(c_X S^G)$ is isomorphic to the singular homology with integral coefficients $H_i(SP^G(Y))$. 
\end{example}

\begin{example}[Ringel duality]
Let $R$ be a commutative ring. The classical Ringel duality functor for representations of Schur algebras can be reformulated in a nice way by using the language of strict polynomial functors. To be more specific, its $n$-fold iteration can be described \cite{TouzeRingel,TouzeArolla} as an equivalence of triangulated categories
$$\Theta^n:\DD(\P_{s,R})\to \DD(\P_{s,R})$$ 
which sends a complex $C$ to the complex $M\mapsto c_{K(M,n)}C[-ns]$, where the brackets denote the suspension, that is $H_i(C[-ns])=H_{ns+i}(C)$. 
\end{example}

\subsection{Functor (co)homology}
Let $R$ be a commutative ring. By functor (co)homology, we mean $\Tor$ or
$\Ext$-groups in $\F_R$, or in related categories. These groups have interpretations in terms of topological Hoschild homology \cite{JP,PW}, and stable homology of classical groups
\cite{Djament,DV}. In this section, we apply our methods to obtain information
$\Ext$-groups, leaving to the reader the completely analogous case of
$\Tor$-groups.

From now on, we let $\A$ be a full abelian subcategory of $\FF_R$, with enough
projectives and injectives, such that for all $M\in \V_R$ the parameterization
functors restrict to functors: $-^M,-_M: \A\to \A$. For example we can take for
$\A$ the category $\FF_R$ itself, or its full subcategory whose objects are the
polynomial functors of degree less or equal to $n$.
Since the parameterization functors are adjoint and exact, 
we have a graded isomorphism, natural with
respect to $F,G\in\A$ and $M\in\V_R$:
$$\Ext^*_{\A}(F^M,G)\simeq \Ext^*_{\A}(F,G_M)\;.$$
So we denote by $\uE^i(F,G)\in\FF_R$ any one of the two
isomorphic functors:
$$M\mapsto \Ext^i_{\A}(F^M,G)\;,\qquad M\mapsto \Ext^i_{\A}(F,G_M)\;.$$

\begin{lemma}\label{lm-applic1}
If $F$ or $G$ is polynomial of degree less or equal to $d$, then $\uE^i(F,G)$ is
also polynomial of degree less or equal to $d$.
If $F$ or $G$ is in $\P_{s,R}$, then $\uE^i(F,G)$ can be viewed as an object of
$\P_{s,R}$.
\end{lemma}
\begin{proof}
Assume for example that $G$ is polynomial of degree less or equal to $d$. Then the statement follows from lemma \ref{lm-cross-effects} with $\C$ the full subcategory of $\DD(\FF_R)$ containing the functors $G_M$, for $M\in\V_R$, and their direct summands as objects, and $\Psi:\C\to R\text{-Mod}$ given by $\Psi(G)=\Ext^i_\A(F,G)$. The proof of the other statements is similar.
\end{proof}

Lemma \ref{lm-applic1} shows that we can apply
theorem \ref{thm-princ} to obtain informations about the integral torsion of the
functor $\uE^i(F,G)$. If we evaluate this functor on $M=R$, we immediately
obtain the following statement. 

\begin{theorem}\label{thm-applic1}
Let $F\in\A$ be a polynomial functor and let $G\in\P_{s,R}$ whose underlying ordinary functor is in $\A$. Assume that $\deg(F)<s$.  
\begin{enumerate}
\item[(i)] For all $i$, $\Ext^i_{\A}(F,G)$ is a torsion abelian group.
\item[(ii)] For all prime $p$, the integral torsion of $p$-primary part of
$\Ext^i_{\A}(F,G)$ is bounded by $p^{r}$ with $r=\big\lceil
\frac{\deg(F)+1-\DIGp(s)}{p-1}\big\rceil$.
\item[(iii)] Assume furthermore  that $F\in\P_{t,R}$.  If $\INT(p,t)\cap \INT(p,s)=\emptyset$, then the $p$-primary part of $\Ext^i_{\A}(F,G)$ is zero.
\end{enumerate}
\end{theorem}

\begin{example}[Mac Lane cohomology with coefficients]
Let us denote by $I$ the inclusion of $\V_R$ in $R\text{-Mod}$. If $F\in\FF_R$,
the extension groups
$$\Ext^*_{\FF_R}(I,F)$$
are called the Mac Lane cohomology of $R$ with coefficients in $F$. Assume that
$F\in\P_{s,R}$ with $s\ge 2$. Theorem \ref{thm-applic1} shows that
$\Ext^*_{\FF_R}(I,F)$ is zero if $s$ is not a power of a prime, and if $s=p^r$
then $\Ext^*_{\FF_R}(I,F)$ is a $\Fp$-vector space.
This result was already obtained in the special cases $F=S^n$ or $F=\Lambda^n$
(for $R=\Z$) in \cite{FP}, where the corresponding extension groups are actually
explicitly computed.
\end{example}

In theorem \ref{thm-applic1}, we can take for $\A$ the category $\FF_R$ itself, or its full subcategory of polynomial functors of degree less or equal to $d$ for a given $d$. Hence theorem \ref{thm-applic1} shows similarities between $\Ext$-computations in these categories. Comparison results of $\Ext$ and $\Tor$ between these categories can be found in \cite{Djament2}.

\begin{remark}
As another common feature between these $\Ext$s, we mention a vanishing result which is well-known in the case $\A=\FF_R$ but which does not seem so well-known when $\A$ is the full subcategory of polynomial functors of degree less or equal to $d$. 
Namely, if $F\in\A$ is a polynomial functor of degree less than $d$, and if $G_1,\dots,G_d$ are reduced functors such that $G_1\otimes\dots\otimes G_d\in\A$, then
$$\Ext^*_\A(F,G_1\otimes\dots\otimes G_d)=0\;.$$
This follows from lemma \ref{lm-applic1} and proposition \ref{prop-cross-implies-vanish}, applied to the functor $\Psi:G\mapsto \Ext^i_\A(F,G)$.
\end{remark}

\subsection{Cohomology of algebraic groups}\label{subsec-cohom}

We refer the reader to \cite{Jantzen} for further details on algebraic groups and the related cohomology.

\begin{theorem}\label{thm-cohom}
Let $G$ be a split connected reductive algebraic group over a PID $R$. Let $M$ be a finitely generated free $R$-module on which $G$ acts rationally. 
Assume that $M$ has a good filtration. Then for all $F\in\P_{s,R}$ and all positive integer $i$,  $H^i(G,F(M))$ is a torsion abelian group whose $p$-primary part is bounded by $p^r$ with $r= \frac{s-\DIGp(s)}{p-1}$.
\end{theorem}
\begin{proof}
Let us fix a positive integer $i$, and let us denote by $\Psi:\P_{s,R}\to R\text{-Mod}$ the functor $\Psi(F)=H^i(G,F(M))$. Since $F$ is a functor of weight $s$, its $s$-th cross effect is an object of $\P_{1,\dots,1,R}(s)$. This category is equivalent to the category of $R$-modules: the equivalence is given by sending a $s$-functor $G$ to the $R$-module $G(R,\dots,R)$, and its inverse sends a $R$-module $K$ to the functor $(M_1,\dots,M_d)\mapsto M_1\otimes\dots\otimes M_d\otimes K$. In particular, for all $F\in\P_{s,R}$, lemma \ref{lm-cross-effects} yields the following formula for the $s$-th cross-effect of $\Psi_F$:
$$\Cr_s(\Psi_F)(M_1,\dots,M_s)\simeq H^i(G,M_1\otimes\dots\otimes M_s\otimes K\otimes M^{\otimes s})\;, \qquad(*)$$
where $K=\Cr_sF(R,\dots,R)$ and the action of $G$ on the coefficient module in the right-hand side is trivial on the factors $M_i$ and $K$.
We claim that the right-hand side of $(*)$ is zero. Hence $\Psi_F$ is a functor of degree less than $s$ and the bound in theorem \ref{thm-cohom} directly follows from theorem \ref{thm-princ}.

Thus, to finish the proof of theorem \ref{thm-cohom}, it remains to prove that our claim that the right-hand side of $(*)$ is zero. Since $M$ has a good filtration, $H^*(G,M^{\otimes s})$ is zero in positive degrees for all $s$. This vanishing implies a slightly more general one, namely that for all $s$ and all $R$-modules $N$ (considered as trivial $G$-modules), $H^*(G,M^{\otimes s}\otimes N)$ is zero in positive degrees. Indeed, the Hochschild complex \cite[I 4.14]{Jantzen} $C(G,M^{\otimes s}\otimes N)$ is isomorphic to $C(G,M^{\otimes s})\otimes N$. Since $C(G,M^{\otimes s})$ is a complex of free $R$-modules whose homology groups are free $R$-modules, we have $H^i(C(G,M^{\otimes s})\otimes N)=H^i(G,M^{\otimes s})\otimes N$ for all $i$. This proves our claim.
\end{proof}

\appendix

\section{Modules over the Schur algebra}

The Schur algebra $S_R(n,s)$ is the endomorphism algebra
$\End_{\Gamma^s(\V_R)}(R^n)$. We identify it with the full subcategory of the
Schur category with $R^n$ as only object. Then the functor $S_R(n,s)\to
\Gamma^s(\V_R)$ induces an evaluation functor 
$$\ev_n:\P_{s,R}\to S_R(n,s)\text{-Mod}\;.$$
It was proved by Friedlander and Suslin \cite[Cor 3.13]{FS} (see also
\cite{SFB,Krause}) that if $n\ge s$, the evaluation functor is an
equivalence of categories.
This section is written for the convenience of the reader. We explain 
how to convert theorem \ref{thm-princ} into an
equivalent statement in the realm of modules over Schur algebras. 

To this purpose, we recall the theory of weights for modules over Schur
algebras. The ground ring $R$ is an arbitrary commutative ring. 
We denote by $D_R(n,s)$ the endomorphism algebra (where $\Gamma^s
(\V_R^{\times n})$ is the category defined in section \ref{subsec-w-vs-d})
$$D_R(n,s)=\End_{\Gamma^s(\V_R^{\times
n})}(\underbrace{R,\dots,R}_{\text{$n$
times}})\;.$$
Thus $D_R(n,s)$ is a free $R$-module with basis $(e_{\lambda})$ indexed by the set 
$\Lambda(n,s)$ of $n$-tuples of nonnegative integers $\lambda=(\lambda_1,\dots,\lambda_n)$ 
with sum $\sum\lambda_i=s$. 
The product satisfies $e_\lambda e_\mu=0$ if $\lambda\ne \mu$ and $e_\lambda^2=e_\lambda$. 
In particular, if $M$ is a $D_R(n,s)$-module then $M_\lambda=e_\lambda M$ is a $D_R(n,s)$-submodule and there is a direct sum decomposition 
$$M=\bigoplus_{\lambda\in\Lambda(n,s)}M_\lambda\;.$$
The $M_\lambda$ are the \emph{weight subspaces of $M$}, and 
the \emph{weights} of $M$ are the tuples $\lambda\in\Lambda(n,s)$
such that $M_\lambda\ne 0$. The number of positive coefficients $\lambda_i$ in a weight $\lambda$ will be called the \emph{length} of the weight. By definition of $D_R(n,s)$, we get examples of
$D_R(n,s)$-modules by evaluating strict polynomial $n$-functors on
$(R,\dots,R)$. We record this fact in the following lemma.
\begin{lemma}\label{lm-record}
Let $F\in\P_{s,R}(n)$. Then $F(R,\dots,R)$ is a $D_R(n,s)$-module. Moreover, if
$F$ is an object of the subcategory
$\P_{(\lambda_1,\dots,\lambda_n),R}(n)\subset \P_{s,R}(n)$ and if
$F(R,\dots,R)\ne 0$, then $(\lambda_1,\dots,\lambda_n)$ is the unique weight of
$F(R,\dots,R)$.
\end{lemma}

The functor $\Gamma^s(\V_R^{\times n})\to \Gamma^s(\V_R)$, $(M_i)\mapsto
\bigoplus M_i$ induces a morphism of algebras:
$$D_R(n,s)\to S_R(n,s)\;.$$
Hence every $S_R(n,s)$-module restricts to a $D_R(n,s)$-module. 
By definition, the weights of a $S_R(n,s)$-module $M$ are the weights of the
corresponding $D_R(n,s)$-module. 

\begin{remark}
The $S_R(n,s)$-modules provide representations of the group scheme $GL_{n,R}$
see e.g. \cite[Section 3]{FS}. From that point of view, the $D_R(n,s)$-module
structure obtained by restriction corresponds to the action of the $n$-torus of
diagonal matrices of $GL_{n,R}$. This explains the notation $D_R(n,s)$. 
\end{remark}

We have the following correspondence between
the degree of a strict polynomial functor $F$ and the weights of the
$S_R(n,s)$-module $F(R^n)$.

\begin{lemma}\label{lm-translate}
Let $F\in\P_{s,R}$ with $s>0$ and let $n\ge s$. The degree of $F$ is equal to
the maximum of the lengths of the weights of 
the $S_R(n,s)$-module $F(R^n)$.
\end{lemma}
\begin{proof}
Let $F\in\P_{s,R}$ of degree less or equal to $d$. There is an isomorphism of
strict polynomial functors with $n$ variables $(M_1,\dots,M_n)$:
$$F(M_1\oplus\dots\oplus M_n)\simeq \bigoplus_{k=1}^d\left(\bigoplus_{j_1<\dots<j_k} \Cr_{k}F(M_{j_1},\dots,M_{j_k})\,\right)\;.$$
Furthermore, by lemma \ref{lm-decomp}, we have a direct sum decomposition $\Cr_k F\simeq \bigoplus (\Cr_k F)_\mu$, with $\mu\in \Lambda(k,s)$ and $(\Cr_k F)_\mu\in \P_{\mu,R}(k)$. So the isomorphism above refines to an isomorphism
$$F(M_1\oplus\dots\oplus M_n)\simeq \bigoplus_{k=1}^d\left(\bigoplus_{\mu\in \Lambda(k,s)}\left(\bigoplus_{j_1<\dots<j_k} (\Cr_{k}F)_{\mu}(M_{j_1},\dots,M_{j_k})\,\right)\right)\;.$$
By evaluation on $(R,\dots,R)$,
we get an isomorphism of $D_R(n,s)$-modules between $F(R^n)$ and a direct sum
of modules of the form $(\Cr_kF)_\mu(R,\dots,R)$. By lemma \ref{lm-record}, the latter is a  $D_R(n,s)$-module of weight $\widetilde{\mu}$, where
$\widetilde{\mu}=(\widetilde{\mu}_1,\dots,\widetilde{\mu}_n)$ with
$\widetilde{\mu}_i=0$ if $i\not\in\{j_1,\dots,j_k\}$, and
$\widetilde{\mu}_{j_\ell}=\mu_\ell$. Thus,
\begin{enumerate}
\item[(i)] if $\Cr_k F(R,\dots,R)$ is nonzero, the lengths of the weights of the $D_R(n,s)$-module $\Cr_k F(R,\dots,R)$ are less or equal to $k$.
\end{enumerate}
Moreover,
\begin{enumerate}
\item[(ii)] if the degree of $F$ is exactly $d$ then 
$\Cr_d F(R,\dots,R)$ is nonzero, and the lengths of the weights of this $D_R(n,s)$-module are all equal to $d$.
\end{enumerate}
Indeed, $\Cr_d F$ is nonzero and additive with respect to each variable. Now lemma \ref{lm-translate} follows from (i) and (ii).
\end{proof}

Now we can use lemma \ref{lm-translate} to translate proposition
\ref{prop-amusette} and theorem \ref{thm-princ} into statements for modules
over the Schur algebras.

\begin{theorem}\label{thm-schur}
Let $R$ be a commutative ring, let $n,s$ be positive integers with $n\ge s$, and
let $M$ be a nonzero $S_R(n,s)$-module.
\begin{enumerate}
\item[(i)] If all the weights of $M$ have length less than $s$, then the
underlying abelian group of $M$ is a torsion group.
\item[(ii)] Assume that the underlying abelian group of $M$ is a $p$-primary
abelian group. Let $d$ be the largest integer such that there exists a weight of $M$ of length $d$.
Then $d\in \INT(p,s)$.
Moreover, if $d<s$ then the integral torsion of the 
underlying abelian group of $M$ is bounded by $p^r$, with $r=\big\lceil
\frac{d+1-\DIGp(s)}{p-1}\big\rceil$.
\end{enumerate}
\end{theorem}


\begin{thebibliography}{99}



\bibitem{ABW} K.~Akin, D.~Buchsbaum, J.~Weyman, Schur functors and Schur complexes.  Adv. in Math.  44  (1982), no. 3, 207--278.


\bibitem{Betley} S.~Betley, Stable derived functors, the Steenrod algebra and
homological algebra in the category of functors. Fund. Math. 168 (2001), no. 3,
279--293.

\bibitem{BM} L.~Breen, R.~Mikhailov, Derived functors of non-additive functors and homotopy theory, Algebr. and Geom. Topology 11, 327-415 (2011) {\tt arXiv 0910.2817}

\bibitem{Bous} A.K.~Bousfield, homogeneous functors and their derived functors, preprint, Brandheis University (1967).





\bibitem{Cartan} H. Cartan, S\'eminaire Henri Cartan de l'\'Ecole Normale sup\'erieure, 1954/1955. Alg\`ebres d'Eilenberg Mac Lane et homotopie, Secr\'etariat math\'ematique, 11 rue Pierre Curie, Paris, 1955 (French). Available online: www.numdam.org





\bibitem{Curtis} E.~Curtis, Some relations between homotopy and homology.  Ann. of Math. (2)  82  1965 386--413. 


\bibitem{Djament} A.~Djament, Sur l'homologie des groupes
unitaires à coefficients polynomiaux. (French) [Homology of unitary groups with
polynomial coefficients] J. K-Theory 10 (2012), no. 1, 87--139.

\bibitem{Djament2} A.~Djament, Groupes d'extensions et foncteurs polynomiaux, J.  Lond. Math. Soc. 2015 ; doi: 10.1112/jlms/jdv017.


\bibitem{DV}  A.~Djament, C.~Vespa, Sur l'homologie des groupes
orthogonaux et symplectiques à coefficients tordus. (French) [Homology of
orthogonal and symplectic groups with twisted coefficients] Ann. Sci. Éc. Norm.
Supér. (4) 43 (2010), no. 3, 395--459.



\bibitem{Do} A.~Dold, Homology of symmetric products and other functors of complexes.  Ann. of Math. (2)  68  1958 54--80.

\bibitem{DP1} A.~Dold, D.~Puppe, Non additive functors, their derived functors and the suspension isomorphism. Proc Nat. Acad. Sci. U.S.A 44 1958 1065--1068.

\bibitem{DP2} A.~Dold, D.~Puppe, Homologie nicht-additiver Funktoren. Anwendungen. (German)  Ann. Inst. Fourier Grenoble  11  1961 201--312.




\bibitem{EML1} S.~Eilenberg, S.~MacLane, Saunders On the groups $H(\Pi,n)$. I. Ann. of Math. (1)  58,  (1953). 55--106.

\bibitem{EML2} S.~Eilenberg, S.~MacLane, Saunders On the groups $H(\Pi,n)$. II. Methods of computation.  Ann. of Math. (2)  60,  (1954). 49--139.

\bibitem{FFSS} V.~Franjou, E.~Friedlander, A.~Scorichenko, A.~Suslin, General linear and functor cohomology over finite fields,   Ann. of Math. (2)  150  (1999),  no. 2, 663--728.

\bibitem{FP} V.~Franjou, T.~Pirashvili, On the Mac Lane cohomology for the ring
of integers. Topology 37 (1998), no. 1, 109--114.


\bibitem{FS} E.~Friedlander, A.~Suslin, Cohomology of finite group schemes over a field,
 Invent. Math. 127 (1997), 209--270.





\bibitem{Jantzen} J.C.~Jantzen, Representations of algebraic groups. Second edition. Mathematical Surveys and Monographs, 107. American Mathematical Society, Providence, RI, 2003. xiv+576 pp. ISBN: 0-8218-3527-0

\bibitem{JM} B.~Johnson, R.~McCarthy, Deriving calculus with cotriple, Trans. Amer. Math. Soc. 356 (2003), 757--803.

\bibitem{JP} M.~Jibladze, T.~Pirashvili, Cohomology of algebraic theories. J. Algebra 137 (1991), no. 2, 253--296.


\bibitem{Krause} H. Krause, Koszul, Ringel and Serre Duality for strict polynomial functors,  Compos. Math. 149 (2013), 996--1018.


\bibitem{LodayValette} J.-L.~Loday, B.~Valette, Algebraic operads. Grundlehren der Mathematischen Wissenschaften [Fundamental Principles of Mathematical Sciences], 346. Springer, Heidelberg, 2012. xxiv+634 pp. ISBN: 978-3-642-30361-6.

\bibitem{Magnus} W.~Magnus, \"Uber Beziehungen zwischen h\"oheren Kommutatoren, J. Reine Angew. Math. 177 (1937), 105--115.

\bibitem{Martin} S.~Martin, Schur algebras and representation theory. Cambridge Tracts in Mathematics, 112. Cambridge University Press, Cambridge, 1993. xvi+232 pp. ISBN: 0-521-41591-8



\bibitem{PW} T.~Pirashvili, F.~Waldhausen, Mac Lane homology and topological Hochschild homology. J. Pure Appl. Algebra 82 (1992), no. 1, 81--98.


\bibitem{PPan} T.~Pirashvili, Introduction to functor homology. Rational representations, the Steenrod algebra and functor homology,  27--53, Panor. Synth\`eses, 16, Soc. Math. France, Paris, 2003.





\bibitem{PRIME} P.~Ribenboim, The new book of prime number records. Springer-Verlag, New York, 1996.



\bibitem{SFB} A.~Suslin, E.~Friedlander, C.~Bendel, Infinitesimal $1$-parameter subgroups and cohomology,  J. Amer. Math. Soc.  10  (1997),  no. 3, 693--728.

\bibitem{TouzeRingel} A.~Touz\'e, Ringel duality and derived functors of nonadditive functors, J. Pure Appl. Algebra 217 (2013), no. 9, 1642--1673.

\bibitem{TouzeArolla} A.~Touz\'e, Applications of functor (co)homology, to appear in the proceedings of the fourth Arolla conference on algebraic topology.

\bibitem{TouzeClassical}A. Touz\'e, Cohomology of classical algebraic groups from the functorial viewpoint, Adv. in Math. 225(1), 33--68, 2010.



\bibitem{Weibel} C.~Weibel, An introduction to homological algebra. Cambridge Studies in Advanced Mathematics, 38. Cambridge University Press, Cambridge, 1994. xiv+450 pp. ISBN: 0-521-43500-5; 0-521-55987-1

\bibitem{Witt} E.~Witt, True Darstellung Liescher Ringe, J. Reine Angew. Math. 177 (1937), 152--160.
\end{thebibliography}
\end{document}